\documentclass[a4paper, 12pt]{amsart}

\usepackage{amsmath}
\usepackage{amsfonts}
\usepackage{amssymb}
\usepackage[utf8]{inputenc}

\newtheorem{thrm}{Theorem}[section]
\newtheorem{col}[thrm]{Corollary}
\newtheorem{lemma}[thrm]{Lemma}
\newtheorem{fact}[thrm]{Proposition}

\newtheorem{question}[thrm]{Question}

\theoremstyle{definition}

\DeclareMathOperator{\supp}{supp}

\newcommand{\NN}{\mathbb{N}}

\title[A metrizable $X$ with $C_p(X)$ not homeomorphic to $C_p(X)\times C_p(X)$]{A metrizable $X$ with $C_p(X)$ not homeomorphic to $C_p(X)\times C_p(X)$}
\author{Miko\l aj Krupski}
\author{Witold Marciszewski}

\address{Institute of Mathematics\\ University of Warsaw\\ \newline Ul. Banacha 2\\02--097 Warszawa\\ Poland }
\email{mkrupski@mimuw.edu.pl}
\email{wmarcisz@mimuw.edu.pl}



\date{\today}
\thanks{The first author was partially supported by the Polish National Science Center research grant UMO-2012/07/N/ST1/03525. The second author was partially supported by the Polish National Science Center research grant DEC-2012/07/B/ST1/03363}

\begin{document}

\begin{abstract}
We give an example of an infinite metrizable space $X$ such that the space $C_p(X)$, of continuous real-valued function on $X$ endowed
with the pointwise topology, is not homeomorphic to its own square $C_p(X)\times C_p(X)$.
The space $X$ is a zero-dimensional subspace of the real line.
Our result answers a long-standing open question in the theory
of function spaces posed by A.V. Arhangel'skii.
\end{abstract}

\maketitle

\section{Introduction}

Let $C_p(X)$ denote the space of all continuous real-valued functions on a Tychonoff space $X$, equipped with the topology of pointwise convergence.
One of the important questions, stimulating the theory of $C_p$--spaces for almost 30 years and leading to interesting examples in this theory,
is the problem whether the space $C_p(X)$ is (linearly) homeomorphic to its own square $C_p(X)\times C_p(X)$, provided $X$ is an infinite compact or
metrizable space, cf. A.V.\ Arhangel’skii's articles \cite[Problem 22]{A1}, \cite[Problem 4]{A2}, \cite[Problem 25]{A3}.
In this note we give a metrizable counterexample to this problem for homeomorphisms.
 
The first nonmetrizable (compact)
counterexamples, i.e. spaces $X$ with $C_p(X)$ not homeomorphic to $C_p(X)\times C_p(X)$, were constructed independently by Gul'ko \cite{G} and Marciszewski \cite{M}. However, the metrizable case seemed to be
more delicate. In \cite{P} R.\ Pol showed that if $M$ is a Cook continuum, then $C_p(M)$ is not linearly homeomorphic to $C_p(M)\times C_p(M)$ (in Section 5 we will show that in fact there is no linear continuous surjection from $C_p(M)$ onto $C_p(M)\times C_p(M)$). He
also gave two other examples of metrizable spaces having the same property: a rigid Bernstein set $B$ and the A.H.\ Stone's set $E$.
This result, settled one part of \cite[Problem 4]{A2} and \cite[Problem 25]{A3} yet the question whether, for a metrizable (compact) space $X$, the space
$C_p(X)$ is always homeomorphic to $C_p(X)\times C_p(X)$ remained open (see \cite[Problem 4.12]{M1}, \cite[Problem 1029]{Pearl}).
It was proved in \cite{vMPP} that if $M$ is a Cook continuum then $C_p(M)$ is not uniformly homeomorphic to $C_p(M)\times C_p(M)$. It is not clear  whether
the notion of uniform homeomorphism in this result can be replaced by a weaker notion of homeomorphism (see \cite[page 656]{vMPP}).

We show that the rigid Bernstein set $B$, considered by R.\ Pol in the context of linear homeomorphisms, can serve as a counterexample solving the problem of Arhangel'skii for homeomorphisms. We shall prove the following:

\begin{thrm}\label{Tw.1}
There exists an infinite zero-dimensional subspace $B$ of the real line (a rigid Bernstein set), such that the function space $C_p(B)$ is not homeomorphic to $C_p(B)\times C_p(B)$.
\end{thrm}

Our proof is based on Theorem \ref{podzialy} below, which is an easy consequence of the main result of \cite{K} proved by the second author. Another important ingredient is Lemma \ref{lemat Witka}
proved in the next section, which may also be of independent
interest. 

The paper is organized as follows. Section 2  introduces basic notation and contains some auxiliary results. In Section 3 we describe the construction of the rigid Bernstein set $B$ and we prove some of its basic properties. Section 4 is devoted to the proof of Theorem \ref{Tw.1}. Finally, Section 5 contains some additional comments and open questions.
\section{Preliminaries}

Let us denote by $\NN$ the set of all positive integers, by $\mathbb{R}$ the set of reals, and by $2^\omega$ the Cantor set. For Tychonoff spaces $X$ and $Y$, by $C_p(X,Y)$ we denote the space of all continuous maps from $X$ into $Y$, endowed with the pointwise convergence topology. For $Y=\mathbb{R}$ we will write $C_p(X)$ rather than $C_p(X,\mathbb{R})$.

For a finite subset $A$ of a space $X$ and for $m\in\NN$ the set
$$O_X(A;\tfrac{1}{m})=\{f\in C_p(X):\forall x\in A\;\;\lvert f(x) \rvert<\tfrac{1}{m}\}$$
is a basic neighborhood of the zero function on $X$ (i.e the constant function equal to zero) in $C_p(X)$
and $\overline{O}_X(A;\tfrac{1}{m})$ is
its closure, i.e. $$\overline{O}_X(A;\tfrac{1}{m})=\{f\in C_p(X):\forall x\in A\;\;\lvert f(x) \rvert\leq\tfrac{1}{m}\}.$$
For a singleton $A=\{x\}$, we will write $\overline{O}_X(x;\tfrac{1}{m})$ rather than $\overline{O}_X(\{x\};\tfrac{1}{m})$.

The following fact is a consequence of results proved by the second author, cf. \cite[proof of Theorem 3.1]{K}.

\begin{thrm}\label{podzialy}
Suppose that $X$ and $Y$ are metrizable spaces.
Let $n\in \NN$ \footnote{In \cite{K} the proof was given for $n=1$ only, but without any changes it works also for arbitrary $n\in \NN$.} and 
suppose that $\Psi:C_p(X)\to C_p(Y)$ is a homeomorphism taking the zero function to the zero function.
Then the space $Y$ can be written as countable union $Y=\bigcup_{r\in\NN}G_r$ of $G_\delta$-subsets such that:
\begin{enumerate}
 \item[(A)] For every $r\in\NN$ there are continuous mappings $f^r_1, \ldots, f^r_{p_r}:G_r\rightarrow X$ and $m\in \NN$ such that
 $\Psi(O_X(A;\tfrac{1}{m}))\subseteq \overline{O}_Y(y;\tfrac{1}{n})$, where $A= \{f^r_1(y),\ldots,f^r_{p_r}(y)\}$.
\end{enumerate}
\end{thrm}

We will need the following lemma.

\begin{lemma}\label{lemat Witka}
Let $X$ and $Y$ be infinite Tychonoff spaces and let $\Psi:C_p(X)\to C_p(Y)$ be a homeomorphism. For any finite set $A\subseteq X$, there exists  
a finite set $B\subseteq Y$, such that, for any $y\in Y\setminus B$ and  $r\in\mathbb{R}$, there is a function $f\in C_p(X)$
such that $f\upharpoonright A=0$,  and $\Psi(f)(y)=r$. 
\end{lemma}

\begin{proof} For a subset $A\subseteq X$, let $C_{p,A}(X)$ denote the subspace $\{f\in C_p(X):\; f\upharpoonright A=0\}$. It is well-known that, for any finite $A\subseteq X$, the space $C_p(X)$ is homeomorphic to the product $\mathbb{R}^A\times C_{p,A}(X)$. Indeed, we have $\mathbb{R}^A = C_p(A)$, and if $T: C_p(A)\to C_p(X)$ is a continuous extension operator (see \cite[6.6.5]{vM}), then the map $\Phi:  C_p(A)\times C_{p,A}(X)\to C_p(X)$ defined by $\Phi(f,g)= T(f)+g$, for $f\in  C_p(A)$ and $g\in  C_{p,A}(X)$, is the required homeomorphism. Observe that $\Phi$ has the property, that 
\begin{equation}\label{factor}
\Phi(f,g)\upharpoonright A= f.
\end{equation}
Fix a finite $A\subseteq X$ and suppose that the assertion of the lemma does not hold true. Then there exist a sequence $(y_n)_{n\in\NN}$ of distinct elements of $Y$ and a sequence $(r_n)_{n\in\NN}$ of reals, such that
\begin{equation}\label{negation}
\Psi(f)(y_n)\ne r_n \quad \mbox{for any } f\in C_{p,A}(X) .
\end{equation}
 Let $\|\cdot\|$ be the Euclidean norm in $\mathbb{R}^A$, $S$ be the unit sphere in $(\mathbb{R}^A,\|\cdot\|)$, and 
$G= \mathbb{R}^A\setminus \{(0,0,\dots,0)\}$. Let $\iota: S\to G$ be the identity embedding. Clearly, the map $\iota$ is not homotopic in $G$ to a constant map. Put
\begin{equation}\label{U}
U = \{e: S\to \mathbb{R}^A:\; e \mbox{ is continuous and } \|e(x)-\iota(x)\|<1  \mbox{ for all } x\in S\}.
\end{equation}
Since any map $e\in U$  is homotopic in $G$ to $\iota$, it is also  not homotopic in $G$ to a constant map. 

Let $\bar{\iota}:S\to \mathbb{R}^A\times C_{p,A}(X)$ be the map defined by $\bar{\iota}(x) = (\iota(x),\mathbf{0})$, for $x\in S$, where $\mathbf{0}$ denotes the zero function in $C_{p,A}(X)$. We put $\tilde{\iota}= \Phi\circ\bar{\iota}:S\to C_p(X)$. 

For a topological space $Z$, by $C(S,Z)$ we denote the space of all continuous maps from $S$ into $Z$, equipped with the compact-open topology.

Let $\pi_1: \mathbb{R}^A\times C_{p,A}(X)\to \mathbb{R}^A$ be the projection onto the first axis. We put $V= \{f\in C(S,\mathbb{R}^A\times C_{p,A}(X)): \pi_1\circ f\in U\}$. Clearly, $V$ is an open subset of $C(S,\mathbb{R}^A\times C_{p,A}(X))$, therefore the set $W=\{\Phi\circ f: f\in V\}$ is an open neighborhood of $\tilde{\iota}$ in $C(S,C_p(X))$. 

Let $D(A)=C_p(X)\setminus C_{p,A}(X)$. 
From property (\ref{factor}) it follows that $\Phi(G\times C_{p,A}(X))= D(A)$. Therefore, one can easily verify that any map $g\in W$ is homotopic in $D(A)$ to $\tilde{\iota}$, hence it is  not homotopic in $D(A)$ to a constant map.

The set $O= \{\Psi\circ g: g\in W\}$ is open in $C(S,C_p(Y))$. Since basic open sets in $C_p(Y)$ depend on finitely many coordinates, we can find a finite set $C\subseteq Y$ such that any $h\in C(S,C_p(Y))$ satisfying
\begin{equation}\label{C}
 h(x)\upharpoonright C = \Psi\circ\tilde{\iota}(x)\upharpoonright C \quad \mbox{for all } x\in S
\end{equation}
belongs to $O$.
Find $y_n\notin C$ and put $D=C\cup\{y_n\}$. Let $ \Theta: \mathbb{R}^D\times C_{p,D}(Y)\to C_p(Y)$ be a homeomorphism such that
\begin{equation}\label{factor2}
\Theta(f,g)\upharpoonright D= f \quad \mbox{for } f\in \mathbb{R}^D, g\in  C_{p,D}(Y),
\end{equation}
 cf. (\ref{factor}). 
 Let $h:S\to \mathbb{R}^D$ be the map defined by
\begin{equation}\label{h}
 h(x)\upharpoonright C = \Psi\circ\tilde{\iota}(x)\upharpoonright C \quad \mbox{and}\quad  h(x)(y_n) = r_n \quad \mbox{for all } x\in S,
\end{equation} 
$\bar{h}: S\to \mathbb{R}^D\times C_{p,D}(Y)$ be defined by $\bar{h}(x) = (h(x),\mathbf{0})$, for $x\in S$, where $\mathbf{0}$ denotes the zero function in $C_{p,D}(Y)$. Finally, we put $\tilde{h}= \Theta\circ\bar{h}:S\to C_p(Y)$. 
  
By (\ref{C}), (\ref{factor2}), and (\ref{h}) we have $\tilde{h}\in O$. Let $r\in C_p(Y)$ be the constant function taking value $r_n$. Consider the homotopy $H: S\times [0,1]\to C_p(Y)$ defined by
\begin{equation}\label{homotopy}
 H(x,t) = (1-t)\tilde{h}(x) + tr \quad \mbox{for } x\in S,
\end{equation}
and joining $\tilde{h}$ with the constant map. Let  $h_t: S\to C_p(Y)$ be defined by $h_t (x)= H(x,t)$.
Observe that, by (\ref{factor2}) and (\ref{h}), for any $t\in [0,1]$ and $x\in S$, $h_t (x)(y_n) = r_n$, hence from (\ref{negation}) it follows that $h_t(S)\subseteq C_p(Y)\setminus \Psi(C_{p,A}(X))$. Therefore the homotopy $\Psi^{-1}\circ H: S\times [0,1]\to C_p(X)$ takes values in $D(A)$ and joins the map $\Psi^{-1}\circ\tilde{h} \in W$ with the constant map $\Psi^{-1}\circ h_1$, a contradiction.
\end{proof}

\section{The rigid Bernstein set $B$}

Let us briefly recall the construction of the rigid Bernstein set $B$ going back to K.\ Kuratowski \cite{Ku}, and used by R.\ Pol in \cite{P}: let $\{(C_\alpha,f_\alpha): \alpha< 2^\omega\}$ be the collection of all pairs $(C,f)$, where $C$ is a copy of the Cantor set in $\mathbb{R}$ and $f:C\to \mathbb{R}$ is a continuous map with uncountable range $f(C)$ disjoint from $C$. We choose inductively distinct points $x_0,y_0,\dots,x_\alpha,y_\alpha,\dots$ with $x_\alpha\in C_\alpha$ and $y_\alpha= f(x_\alpha)$, and we put $B = \{x_\alpha: \alpha< 2^\omega\}$. A more detailed description of this construction can be found in \cite[Example 6.13.1]{vM}.

Recall that a space $X$ is a Baire space if the Baire Category Theorem holds  for $X$, i.e.\ every sequence $(U_n)$ of dense open subsets of $X$ has a dense intersection in $X$.

\begin{lemma}\label{Bernstein1} Each $G_\delta$-subspace of $B$ is a Baire space.
\end{lemma}

\begin{proof} Observe that $B$ is a Bernstein set, i.e. both $B$ and $\mathbb{R}\setminus B$ intersect each copy of the Cantor set in $\mathbb{R}$ (see \cite[Example 6.13.1, Claim 1]{vM}).  Since any dense $G_\delta$-subspace of a Baire space is also a Baire space, it is enough to show the thesis of the lemma for closed subsets of $B$.
Let $G$ be a closed subset of $B$.
By Hurewicz theorem, in order to prove that $G$ is a Baire space, it is enough to check that $G$ does not contain a closed copy of the rationals (see, \cite[Theorem 1.9.12]{vM}). Striving for a contradiction, suppose that $Q$ is a closed subset of $G$ homeomorphic to the rationals. Then the closure $\overline{Q}$ of $Q$ in $\mathbb{R}$ is a perfect subset of $\mathbb{R}$, hence its uncountable. Then the set $\overline{Q}\setminus Q\subseteq \mathbb{R}\setminus B$ is an uncountable  $G_\delta$-set in $\mathbb{R}$, therefore it contains a copy of the Cantor set disjoint from $B$, a contradiction.
\end{proof}

For our purposes we will need a stronger version of rigidity of $B$ than used in \cite{P}.

\begin{lemma}\label{Bernstein2} If $G$ is a nonempty $G_\delta$-subset of  $B$,  then each continuous function $f:G\to B$ is either the identity or is constant on a nonempty relatively open subset of $G$.
\end{lemma}

\begin{proof} Our argument is a slight modification of the proof of Claim 3 in \cite[Example 6.13.1]{vM}. 

Suppose that $f$ is not the identity. Then we can find a nonempty relatively open subset $U$ of $G$ such that the closures (taken in $\mathbb{R}$) $\overline{U}$ and $\overline{f(U)}$ are disjoint. We will show that $f(U)$ is countable. Assume towards a contradiction that this is not the case. By Lavrentiev theorem $f\upharpoonright U$ can be extended to a continuous function $\tilde{f}: S\to \overline{f(U)}$, where $S$ is a $G_\delta$-subset of $\overline{U}$. Since $U$ is a $G_\delta$-subset of $B$, we can assume (shrinking $S$, if necessary) that $S\cap B = U$. The image $\tilde{f}(S)$ is uncountable since it contains $f(U)$. Therefore, by \cite[Theorem 1.5.12]{vM} there exists a copy $C$ of the Cantor set in $S$ such that $\tilde{f}$ is one-to-one on $C$. Then there exists $\alpha< 2^\omega$ such that $(C,\tilde{f}\upharpoonright C) = (C_\alpha,f_\alpha)$. Hence $x_\alpha\in C\cap B\subseteq S\cap B = U$, so $y_\alpha = f_\alpha(x_\alpha) = \tilde{f}\upharpoonright C (x_\alpha) = f(x_\alpha)\in B$, which is a contradiction with the construction of $B$.

Now, for every $t$ in the countable set $f(U)$, let $A_t = \{x\in U: f(x)=t\}$. Then $\{A_t: t\in f(U)\}$ is a countable cover of $U$ by relatively closed sets. By Lemma \ref{Bernstein1} $U$ is a Baire space, therefore one of the sets $A_t$ has a nonempty interior in $U$, hence also in $G$.
\end{proof}

\begin{col}\label{sztywnosc} If $G$ is an uncountable $G_\delta$-subset of  $B$,  then for  each continuous function $f:G\to B$ there exists an uncountable $G_\delta$-subset $G'$ of $G$ such that the restriction $f\upharpoonright G'$ is either the identity or is constant.
\end{col}

\begin{proof} Let $\mathcal{U}$ be a countable base in $B$, and let $V=\bigcup\{U\in\mathcal{U}: U\cap G \mbox{ is countable} \}$. Then $H = G\setminus V$ is a nonempty $G_\delta$-subset of $B$, and each nonempty open subset of $H$ is uncountable. It remains to apply Lemma \ref{Bernstein2} for $H$ and $f\upharpoonright H$.
\end{proof}

\section{Proof of Theorem \ref{Tw.1}}

Strengthening a result from \cite{P}, we shall prove that the spaces $C_p(B)$ and $C_p(B)\times C_p(B)$
are not homeomorphic. Of course $C_p(B)\times C_p(B)$ is linearly homeomorphic to $C_p(B\oplus B)$, where
$B\oplus B$ is a discrete sum of two copies of $B$ and thus can be viewed as $B\times\{1,2\}$.

It will be convenient to use the following notation:
$$A_i=A\times \{i\}\subseteq B\oplus B\text{, }i=1,2,$$
for a subset $A\subseteq B$.
Similarly, $x_i=(x,i)\in B\oplus B$, for any $x\in B$.
Thus $A_i$ is a copy of $A$ lying in the corresponding copy of $B$ in the space $B\oplus B$.

\medskip

Striving for a contradiction, suppose that there is a homeomorphism $$\Phi:C_p(B)\to C_p(B \oplus B).$$
It is clear that without loss of generality we can assume that $\Phi$ takes the zero function to the zero function.

From Theorem \ref{podzialy} (applied with $n=1$, $X=B\oplus B$, $Y=B$ and $\Psi=\Phi^{-1}$) it follows that
$B$ is a countable union of $G_\delta$ subsets $G_r$ satisfying property (A). Fix $r$ such that $G_r$ is uncountable and consider finitely many continuous functions
$f'_1,\ldots , f'_{p'}:G_r\to B\oplus B$ provided by Theorem \ref{podzialy}.

For any $j\leq p'$ and $i\in\{1,2\}$ the set $(f'_j)^{-1}(B_i)$ is open in $G_r$. Thus, for any uncountable $G_\delta$ subset $G\subseteq G_r$ and any $j\leq p'$,
there is an uncountable $G_\delta$ subset $G'\subseteq G$ with $f'_j(G')\subseteq B_1$ or $f'_j(G')\subseteq B_2$.
Applying this observation successively, for $j=1,\ldots,p'$, we can find an uncountable $G_\delta$ set $H\subseteq G_r$
such that $f'_j(H)\subseteq B_1$ or $f'_j(H)\subseteq B_2$, for $j\leq p'$. 

By Corollary \ref{sztywnosc}, there is an uncountable $G_\delta$ set $H^1\subseteq H$, such that the function $f'_1\upharpoonright H^1$ is either the identity
(up to identification of $H^1_i$ with $H^1$)
or is constant. Applying Corollary \ref{sztywnosc} recursively, we can construct a decreasing sequence $G_r\supseteq H \supseteq H^1\supseteq \ldots \supseteq H^{p'}$
of uncountable $G_\delta$ subsets of $B$
such that, for $j\leq p'$, $f'_j\upharpoonright H^j$ is either the identity or is constant. Indeed, if $H^j$ is constructed, where $j<p'$,
we consider the function $f'_{j+1}\upharpoonright H^j$ and apply Corollary \ref{sztywnosc} to find a desired uncountable $G_\delta$ set $H^{j+1}\subseteq H^j$.

If $C'=H^{p'}$ then
each $f'_j\upharpoonright C'$ is either the identity (up to identification of $C'_i$ with $C'$) or is constant.
Hence, there is a finite set $J'\subseteq B\oplus B$ such that
$$\{f'_1(x), \ldots, f'_{p'}(x)\}\subseteq \{x_1,x_2\}\cup J' \text{, for any } x\in C'.$$

Property (A) from Theorem \ref{podzialy} implies that there is $k\in \NN$ such that
\begin{align}\label{*}
\Phi^{-1}(O_{B\oplus B}(\{x_1,x_2\}\cup J';\tfrac{1}{k}))\subseteq \overline{O}_B(x;1) \text{, for any }x\in C'.
\end{align}

Now, applying Theorem \ref{podzialy} once more (with $n=2k$, $X=B$, $Y=B\oplus B$ and $\Psi=\Phi$) together with Corollary \ref{sztywnosc} (applied recursively
as before),
we can find an uncountable $G_\delta$ set $C\subseteq C'$ and finitely many continuous functions (being the restriction of functions provided by property (A) from Theorem
\ref{podzialy})
$f^1_1,\ldots , f^1_p:C_1\to B$ and
$f^2_1,\ldots , f^2_q:C_2\to B$ such that
each $f^1_i, f^2_i$ is either the identity (up to identification of $C_i$ with $C$) or is constant.
Hence, there is a finite set $J\subseteq B$ such that
$$\{f^1_1(x_1), \ldots, f^1_{p}(x_1)\}\cup \{f^2_1(x_2), \ldots, f^2_{q}(x_2)\}\subseteq \{x\}\cup J \text{, for any } x\in C.$$
Property (A) from Theorem \ref{podzialy} implies that there is $m\in\NN$ such that
\begin{align}\label{**}
\Phi(O_{B}(\{x\}\cup J;\tfrac{1}{m}))\subseteq \overline{O}_{B\oplus B}(\{x_1,x_2\};\tfrac{1}{2k}) \text{, for any }x\in C.
\end{align}

By the continuity of $\Phi^{-1}$, there is a finite set $I\subseteq B\oplus B$ and $\varepsilon>0$ such that
\begin{align}\label{zbior I}
\Phi^{-1}(O_{B\oplus B}(I;\varepsilon))\subseteq O_B(J;\tfrac{1}{m}).
\end{align}
 
By Lemma \ref{lemat Witka} (where $X=B\oplus B$, $Y=B$, $\Psi=\Phi^{-1}$, $A=I\cup J'$) there are $v_1, v_2\in C_p(B\oplus B)$ and
$c\in C$ such that

\begin{enumerate}
 \item[(i)]$c_1, c_2\notin I\cup J'$,
 \item[(ii)]$v_1\upharpoonright(I\cup J')=0$, $v_2\upharpoonright(I\cup J')=0$,
 \item[(iii)]$\Phi^{-1}(v_1)(c)>2$, $\Phi^{-1}(v_2)(c)<-2$.
\end{enumerate}

\medskip

\textbf{Claim 1.} $|v_i(c_1)|\geq\tfrac{1}{k}$ or $|v_i(c_2)|\geq\tfrac{1}{k}$, for $i=1,2$.
\begin{proof}
If not, then by (ii) and \eqref{*} we would have $|\Phi^{-1}(v_i)(c)|\leq 1$, contradicting (iii).
\end{proof}
So let $i,j\in\{1,2\}$ be such that
\begin{align}\label{i,j}
|v_1(c_i)|\geq\tfrac{1}{k} \text{ and } |v_2(c_j)|\geq\tfrac{1}{k}.
\end{align}
%
%
We shall consider two cases:

\medskip

\textit{Case 1.} 
$v_1(c_1)\cdot v_2(c_2)=v_1(c_2)\cdot v_2(c_1).$
Let
\begin{align*}\label{j'}
j'=j+1 \mod 2.
\end{align*}
By the continuity of $\Phi^{-1}$, there is $\delta>0$ and $h\in C_p(B\oplus B)$ such that

\begin{equation}\label{funkcja h}
  \left\{\begin{aligned}
  &   h(c_{j'})=v_2(c_{j'})+\delta,\\
&h\upharpoonright (I\cup J'\cup \{c_j\})=v_2\upharpoonright(I\cup J'\cup \{c_j\}),\\
&\Phi^{-1}(h)(c)<-1. 
  \end{aligned}
 \right. 
\end{equation}
We put $u_1=v_1$ and $u_2=h$. Using \eqref{i,j} one can easily verify that $v_1(c_j)\ne 0$, hence
 $$u_1(c_1)\cdot u_2(c_2)\neq u_1(c_2)\cdot u_2(c_1).$$

\medskip

\textit{Case 2.}
$v_1(c_1)\cdot v_2(c_2)\neq v_1(c_2)\cdot v_2(c_1).$ Then we put $u_1=v_1$, $u_2=v_2$.

\bigskip

We define the mapping $\varphi:\mathbb{R}\times\mathbb{R}\to \mathbb{R}$ by the formula
$$\varphi(t_1,t_2)=\Phi^{-1}\Big(t_1 u_1+t_2 u_2\Big)(c),$$
i.e. $\varphi$ is the composition of the mapping $(t_1,t_2)\mapsto t_1 u_1+t_2 u_2$ with $\Phi^{-1}$
and the evaluation functional at $c$.
Consider
$$Z=\{(t_1,t_2)\in \mathbb{R}\times\mathbb{R}: |t_1 u_1(c_1)+t_2 u_2(c_1)|\geq \tfrac{1}{k}\; \text{ or }\; |t_1 u_1(c_2)+t_2 u_2(c_2)|\geq \tfrac{1}{k}\}.$$

Let
\begin{gather*}
 m_1=\{(t_1,t_2)\in\mathbb{R}\times\mathbb{R}:\;t_1 u_1(c_1)+t_2 u_2(c_1)=\tfrac{1}{k}\},\\
 m_2=\{(t_1,t_2)\in\mathbb{R}\times\mathbb{R}:\;t_1 u_1(c_2)+t_2 u_2(c_2)=\tfrac{1}{k}\}.
\end{gather*}
Note, that from the definition of $u_1$ and $u_2$ it follows that the above sets are nonempty,
i.e. it can not happen that $u_1(c_1)=u_2(c_1)=0$ or $u_1(c_2)=u_2(c_2)=0$. Hence $m_1$ and $m_2$ are non-parallel lines.
Indeed, by the definition of $u_1$ and $u_2$, cf. Case 1 and Case 2, we have
$u_1(c_1)\cdot u_2(c_2)\neq u_1(c_2)\cdot u_2(c_1)$, which means exactly that $m_1$ and $m_2$ are not parallel.

Since $m_1$ and $m_2$ are not parallel,
the set $Z$ is connected (being the plane with a parallelogram removed).

\medskip

\textbf{Claim 2.} $\varphi(Z)\subseteq \mathbb{R}\setminus (-\tfrac{1}{m},\tfrac{1}{m})$.
\begin{proof}
Otherwise, by (ii), \eqref{zbior I} and \eqref{funkcja h}
$$\Phi^{-1}\Big(t_1 u_1+t_2 u_2\Big)\in O_M(\{c\}\cup J;\tfrac{1}{m}),$$
for some $(t_1, t_2)\in Z$.
Hence \eqref{**} implies that
\begin{gather*}
|t_1 u_1(c_1)+t_2 u_2(c_1)|\leq \tfrac{1}{2k}<\tfrac{1}{k},\\
|t_1 u_1(c_2)+t_2 u_2(c_2)|\leq \tfrac{1}{2k}<\tfrac{1}{k}.
\end{gather*}
However this contradicts the definition of $Z$.
\end{proof}

By \eqref{i,j} and \eqref{funkcja h}, we have
$(1,0), (0,1)\in Z$. Further, by (iii) and \eqref{funkcja h}, we infer that
$$\varphi(1,0)=\Phi^{-1}(u_1)(c)>2,\quad \varphi(0,1)=\Phi^{-1}(u_2)(c)<-1.$$
This means that $\varphi(Z)\cap(-\infty,-\tfrac{1}{m})\neq\emptyset$, $\varphi(Z)\cap(\tfrac{1}{m}, \infty)\neq\emptyset$ and,
by Claim 2, $\varphi(Z)\cap (-\tfrac{1}{m},\tfrac{1}{m})=\emptyset$. Therefore the set $\varphi(Z)$ is not connected,
a contradiction with connectedness of $Z$. This ends the proof of Theorem \ref{Tw.1}.

\section{Remarks and problems}

The following question of Arhangiel'skii remains open.

\begin{question} Let $X$ be an infinite compact metrizable space. 
Is it true that $C_p(X)$ is homeomorphic to $C_p(X)\times C_p(X)$?
\end{question}

A natural candidate for a counterexample is the Cook continuum $M$ used in \cite{P} and \cite{vMPP} in the context of linear and uniform homeomorphisms.

\subsection{Continuous surjections}


The following old question of Arhangiel'skii is also related to the problem considered in this paper, cf. \cite[Problem 5]{A2}

\begin{question}\label{q1}(Arhangel'skii)
	Is it true that $C_p(X)$ can always be continuously mapped onto its own square $C_p(X)\times C_p(X)$?
\end{question}

Though the above question is open, the affirmative answer is known for some particular classes of spaces such as zero-dimensional compacta, cf. \cite{doktorat}, \cite{Okunev}
or metrizable compact spaces, cf. \cite{doktorat}. For the reader's convenience below we give short proofs of these facts.

\begin{fact}
 If $X$ is a compact zero-dimensional space, then $C_p(X)\times C_p(X)$ is a continuous image of $C_p(X)$. 
\end{fact}

\begin{proof} For the purpose of this proof we will identify the square $C_p(X)\times C_p(X)$ with the space $C_p(X,\mathbb{R}^2)$. For any $n\in\mathbb{N}$, let $B_n=[-n,n]^2\subseteq \mathbb{R}^2$. By \cite[Lemma 1]{Okunev} there exists a continuous map $\phi_n: 2^\omega\to B_n$ such that
\begin{align}\label{Ok1}
(\forall f\in C_p(X,B_n))\,(\exists g\in C_p(X,2^\omega))\quad f = \phi_n\circ g
\end{align}
(let us note that the key ingredient of the proof of this lemma is the Marde\v{s}ic factorization theorem \cite{Mardesic}). We define $\phi: \mathbb{N}\times 2^\omega\to \mathbb{R}^2$ by 
\begin{align}\label{Ok2}
\phi(n,x) = \phi_n(x)\quad \mbox{for } n\in\mathbb{N}, x\in 2^\omega.
\end{align}
Let $e$ be a homeomorphism of $\mathbb{N}\times 2^\omega$ onto a closed subset $A$ of $\mathbb{R}$, and let $\psi:\mathbb{R}\to\mathbb{R}^2$ be a continuous extension of the composition $\phi\circ e^{-1}: A\to\mathbb{R}^2$. 

Now, we can define the map $\varphi: C_p(X)\to C_p(X,\mathbb{R}^2)$ as follows
\begin{align}\label{Ok3}
\varphi(f) = \psi\circ f\quad \mbox{for } f\in C_p(X).
\end{align}
Clearly, $\varphi$ is continuous, so it remains to check that it is surjective. Take any $h\in C_p(X,\mathbb{R}^2)$. By compactness of $X$, the image $h(X)$ is contained in some $B_n$. From (\ref{Ok1}) and (\ref{Ok2}), and it follows that there exists a continuous $g: X\to \{n\}\times 2^\omega$ such that $h = \phi\circ g$. By the properties of the maps $e$ and $\psi$, we have $h = \psi\circ e\circ g$, therefore $h = \varphi(e\circ g)$.
\end{proof}

\begin{fact}
 If $X$ is a compact metrizable space, then $C_p(X)\times C_p(X)$ is a continuous image of $C_p(X)$. 
\end{fact}

\begin{proof}
Clearly, it is enough to consider the case of infinite space $X$. Let $(x_n)_{n=1}^\infty$ be sequence of distinct points of $X$ converging to a point $x_0$,
and let $S=\{x_n: n=0,1,\dots\}$. The space $C_p(S)$ is a Borel subset of $\mathbb{R}^S$ which is not $\sigma$-compact, cf. \cite[Theorems 6.3.6 and 6.3.10]{vM}.
Hence, from Hurewicz theorem (see, \cite[Theorem 21.18]{Ke}) it follows that $C_p(S)$ contains a closed copy $P$ of the space of irrationals.
Since the Banach space $C(X)$ is separable, there is a continuous map $h$ of $P$ onto $C(X)\times C(X)$ and its continuous extension
$H: C_p(S)\to C(X)\times C(X)$. Obviously, $H$ is also continuous with respect to the weaker pointwise topology in $C(X)\times C(X)$.
To finish the proof, it remains to observe that the restriction operator $f\mapsto f\upharpoonright S$, for $f\in C_p(X)$,
is a continuous surjection of $C_p(X)$ onto $C_p(S)$.
\end{proof}

On the other hand, as we shall prove, $C_p(X)\times C_p(X)$ is not always a \textit{linear} continuous image of $C_p(X)$, even for a
(compact) metrizable $X$. A Cook continuum $M$ or a rigid Bernstein set $B$ can serve as an example. Let us recall that a Cook continuum is a nontrivial metrizable continuum $M$
such that, for every subcontinuum $C\subseteq M$, every continuous mapping $f:C\rightarrow M$ is either the identity or is constant.

The following proposition strengthens slightly results of R.\ Pol \cite[Theorem 3.1]{P} and \cite[Theorem 4.1]{P}.

\begin{fact}
If $X=M$ or $X=B$, then there is no linear continuous surjection from $C_p(X)$ onto $C_p(X)\times C_p(X)$.
\end{fact}

\begin{proof}
We will give a proof for $X=M$ only. The case $X=B$ is almost the same (the role of nontrivial subcontinua in the argument below is played by uncountable $G_\delta$-subsets).
Striving for a contradiction, suppose that $\varphi:C_p(M)\to C_p(M\oplus M)$ is a linear continuous surjection.
Similarly as in Section 4, we view $M\oplus M$ as $M\times \{1,2\}$ and by $x_i, A_i$ we denote the copies of $x\in M$, $A\subseteq M$ in $M\times\{i\}$.
It is well-known (see \cite[Ch. 6.8]{vM}) that to each $y\in M\oplus M$ we can assign a nonempty finite set $\supp_{\varphi}(y)$ such that
\begin{align}
&\text{the assignment }y\mapsto \supp_{\varphi}(y) \text{ is lower-semicontinuous,}\\
&\varphi(f)(y)=\sum_{z\in \supp_{\varphi}(y)}\lambda(y,z)f(z), \quad \mbox{for some } \lambda(y,z)\in\mathbb{R}. \label{support}
\end{align}
Applying \cite[Lemma 6.13.2]{vM}, we can find a nonempty open subset $U_1\subseteq M_1$ and, for some $n\in \NN$, continuous mappings
$s_i:U_1\to M$, $i\leq n$,
such that $$\supp_{\varphi}(y)=\{s_1(y),\ldots,s_n(y)\}$$ for every $y\in U_1$.
By Janiszewski theorem (see \cite[\S 47.III.1]{Ku1})
, there is a nontrivial continuum $C_1\subseteq U_1$. By the rigidity of $M$ the restriction of each mapping $s_i$ to $C_1$ is
either the identity (up to identification of $C_1$ with $C$) or is constant. Hence, there is a finite set $J\subseteq M$ such that
\begin{align}\label{C1}
\supp_{\varphi}(y_1)=\{y\}\cup J,\quad \text{ for } y_1\in C_1.
\end{align}

Using the same argument as above for $C_2$ (the copy of $C_1$ in $M_2$) instead of $M_1$, we get a nontrivial continuum $K_2\subseteq C_2$ and a finite set $I\subseteq M$
such that
\begin{align}\label{K2}
\supp_{\varphi}(y_2)=\{y\}\cup I, \quad \text{ for } y_2\in K_2.
\end{align}

Put $k=|I\cup J|$ and 
let $A=\{a^1,\ldots,a^{k+1}\}\subseteq K\setminus (I\cup J)$ be a set of cardinality $k+1$.
By \eqref{C1}, \eqref{K2} we have
\begin{align}\label{A1A2}
\bigcup\{\supp_{\varphi}(y):y\in A_1\cup A_2\}\subseteq A\cup I \cup J.
\end{align}

Note that $|A_1\cup A_2|=2k+2$ and $|A\cup I \cup J|=2k+1$.
We have $A\cup I\cup J=\{x^1,\ldots, x^{2k+1}\}$, for some $x^i\in M$, $i\leq 2k+1$.

Now, we define a mapping $T:\mathbb{R}^{2k+1}\to \mathbb{R}^{2k+2}$ in the following way (cf. \cite[page 51]{P}, \cite[page 451]{vM}). Given $(r^1,\ldots , r^{2k+1})\in \mathbb{R}^{2k+1}$
choose a function $f\in C_p(M)$ such that $f(x^i)=r^i$.
Let $$T(r^1\ldots, r^{2k+1})=(\varphi(f)(a^1_1),\ldots, \varphi(f)(a^{k+1}_1),\varphi(f)(a^1_2),\ldots, \varphi(f)(a^{k+1}_2)).$$
Formula \eqref{support} and \eqref{A1A2} imply that $T$ does not depend on the choice of a function $f$ (cf. \cite[Lemma 6.8.1]{vM}) and hence
$T$ is well defined. Obviously, it is also linear. Since $\varphi$ is a surjection, one can easily verify that $T$ is onto.
However, this is a contradiction since a linear mapping cannot raise dimension.
\end{proof}

\subsection{Countable spaces}

It is well-known that, for any countable metrizable nondiscrete spaces $X$ and $Y$, the function spaces $C_p(X)$ and $C_p(Y)$ are homeomorphic,
see \cite{DMM}, \cite{vM}. Therefore, if $X$ is an infinite countable metrizable space, then $C_p(X)$ is clearly homeomorphic to $C_p(X) \times C_p(X)$
(since $C_p(X) \times C_p(X)$ can be identified with $C_p(X\oplus X)$). However, it is not clear what happens if we drop the metrizability assumption:

\begin{question}
 Let $X$ be an infinite countable space. Is it true that $C_p(X)$ is homeomorphic to $C_p(X)\times C_p(X)$?
\end{question}

The linear topological classification of $C_p(X)$--spaces for countable metrizable spaces $X$ is not fully understood; see Baars and de Groot \cite{BG}.
In particular, the following question seems to be open:

\begin{question}
	Suppose that $X$ is an infinite countable metrizable space. Is it true that $C_p(X)$ is linearly (uniformly) homeomorphic to $C_p(X)\times C_p(X)$?
\end{question}

Let us note that from results of Baars and de Groot (\cite{BG},  \cite[Theorem 3.22]{BGFund}) it follows that the above question has the affirmative answer if $X$ is either non-scattered or is scattered of height $\leq \omega$.

It is known that for an infinite Polish zero-dimensional space $X$ which is either compact or not $\sigma$-compact, the space $C_p(X)$ is linearly homeomorphic to
$C_p(X)\times C_p(X)$ (see \cite{A4} and \cite{BG}); therefore it is clear that a metrizable space $X$ such that $C_p(X)$ is not homeomorphic to $C_p(X)\times C_p(X)$
cannot be simultaneously compact and zero-dimensional. It is natural to ask what happens if a Polish
zero-dimensional space
$X$ is $\sigma$-compact:

\begin{question}
 Suppose that $X$ is a Polish zero-dimensional $\sigma$-compact space. Is it true that $C_p(X)$ is (linearly) homeomorphic to $C_p(X)\times C_p(X)$?
\end{question}

\end{document}